\documentclass[a4paper,12pt]{amsart}

\usepackage{latexsym}
\usepackage{amssymb,amsmath,amsthm}
\usepackage{graphicx}
\usepackage{url}

\newtheorem{theorem}{Theorem}[section]

\newtheorem{corollary}[theorem]{Corollary}
\newtheorem{proposition}[theorem]{Proposition}

\newtheorem{lemma}[theorem]{Lemma}

\newtheorem*{utheorem}{Theorem}

\newcommand{\R}{\mathbf{R}}
\newcommand{\N}{\mathbf{N}}

\newcommand{\e}{\mathrm{e}}
\renewcommand{\P}{\mathbf{P}}

\DeclareMathOperator{\Cent}{C}
\DeclareMathOperator{\Sym}{Sym}
\DeclareMathOperator{\cp}{cp}
\newcommand{\E}{\mathbf{E}}

\renewcommand{\r}{\mathbf{r}}

\newcommand{\kSn}{\kappa(S_n)}

\newcounter{thmlistcnt}
\newenvironment{thmlist}%
    {\setcounter{thmlistcnt}{0}%
    \begin{list}{\emph{(\roman{thmlistcnt})}}{%
        \usecounter{thmlistcnt}%
        \setlength{\topsep}{0pt}%
        \setlength{\leftmargin}{0pt}%
        \setlength{\itemsep}{0pt}%
        \setlength{\itemindent}{28pt}}%
    }%
    {\end{list}}%

\renewcommand{\succeq}{\succcurlyeq}

\begin{document}

\title[The Probability that Group Elements are Conjugate]{The Probability that a Pair of Elements of a Finite Group are Conjugate}


\author{Simon R. Blackburn}
\address{Department of Mathematics, Royal Holloway, University of London, Egham, Surrey TW20 0EX, United Kingdom}
\email{s.blackburn@rhul.ac.uk}

\author{John R. Britnell}
\address{Department of Mathematics, University of Bristol, University Walk, Bristol BS8 1TW, United Kingdom}
\email{j.r.britnell@bristol.ac.uk}
\thanks{The research of the second author
was supported  by the Heilbronn Institute for Mathematical Research}

\author{Mark Wildon}
\address{Department of Mathematics, Royal Holloway, University of London, Egham, Surrey TW20 0EX, United Kingdom}
\email{mark.wildon@rhul.ac.uk}



\date{24 February 2012}


\keywords{Conjugacy, probability in finite groups, symmetric group, commuting conjugacy classes\vspace{1.5pt}
}

\begin{abstract}
Let $G$ be a finite group, and let $\kappa(G)$ be the probability that
elements $g$, $h\in G$ are conjugate, when $g$ and $h$ are chosen
 independently and uniformly at random. The paper classifies those
groups $G$ such that $\kappa(G)\geq 1/4$, and shows that $G$ is
abelian whenever $\kappa(G)|G| <7/4$. It is also
shown that $\kappa(G)|G|$ depends only on the isoclinism class of $G$.

Specialising to the symmetric group $S_n$, the paper shows that
$\kSn \leq C/n^2$ for an explicitly determined constant $C$. This
bound leads to an elementary proof of a result of Flajolet \emph{et al},
that $\kSn \sim A/n^2$ as
$n\rightarrow \infty$ for some constant $A$. The same techniques
provide analogous results for $\rho(S_n)$, the probability that
two elements of the symmetric group have conjugates that commute.
\end{abstract}

\maketitle

\thispagestyle{empty}

\section{Introduction}
\label{sec:introduction}

Let $G$ be a finite group. We define $\kappa(G)$ to be the probability
that two elements $g$, $h\in G$ are conjugate, when $g$ and $h$ are
chosen independently and uniformly at random from $G$. Let
$g_1,g_2,\ldots,g_k$ be a complete set of representatives for the
conjugacy classes of $G$. It is easy to see that
\begin{equation}
\label{eq:kappa}
\kappa(G)=\frac{1}{|G|^2}\sum_{i=1}^k|g_i^G|^2 = \sum_{i=1}^k\frac{1}{|\Cent_G(g_i)|^2}
\end{equation}
where $C_G(g)$ denotes the centralizer of an element $g \in G$.

This paper is divided into two parts. In the first part of the paper,
we begin a study of $\kappa(G)$ by proving two `gap' results
which classify the groups
for which $\kappa(G)$ is unusually small or large.
We also show that $\kappa(G)|G|$ is an invariant of the isoclinism
class of $G$.
In the second part, we prove bounds and find the asymptotic
behaviour of
$\kSn $ where~$S_n$ is the symmetric group of degree $n$.
Our techniques allow us to prove similar results on the
probability $\rho(S_n)$ that two elements of~$S_n$,
chosen independently and uniformly at random,
have conjugates that
commute. We end the paper with some
further remarks and open problems
on the behaviour of $\kappa(G)$ and $\rho(G)$ when $G$ is an arbitrary finite group

We now describe our results in more detail.

\subsection{Results on general finite groups}

It is clear that $\kappa(G) \ge 1/|G|$ and that equality holds
exactly when $G$ is abelian.
We prove the following `gap' result.

\begin{theorem}\label{thm:lowergap}
Let $G$ be a finite group. 
If $\kappa(G) < \frac{7}{4|G|}$
then $G$ is abelian. Moreover, \hbox{$\kappa(G) = \frac{7}{4|G|}$} if
and only if the index of the centre~$Z(G)$ in~$G$ is~$4$.
\end{theorem}

This theorem is proved by an elementary counting argument. There are
many groups whose centres have index $4$, for example
any group of the form $D_8 \times A$
where $D_8$ is the dihedral group of order $8$ and $A$ is
abelian has this property. The observation
that the groups with this property form a single
isoclinism class motivates our second theorem.

\begin{theorem}\label{thm:isoclinism}
If $G$ and $H$ are isoclinic finite groups then $\kappa(G)|G| = \kappa(H)|H|$.
\end{theorem}

Theorem~\ref{thm:isoclinism} implies that for each isoclinism class $I$ there exists a constant $b_I$ such that $\kappa(G)=b_I/|G|$ for all groups $G\in I$. We provide a definition of isoclinism
and set up the relevant notation immediately before the proof of
Theorem~\ref{thm:isoclinism} in \S\ref{sec:isoclinism} below.

Our third main theorem
classifies the groups $G$ such that $\kappa(G)$ is
large.

\begin{theorem}\label{thm:uppergap}
Let $G$ be a non-trivial finite group. Then $\kappa(G) \ge 1/4$ if and
only if one of the following holds:
\begin{thmlist}
\item $|G| \le 4$;
\item $G \cong A_4$, $S_4$, $A_5$ or $C_7 \rtimes C_3$;
\item \label{scinv} $G \cong A \rtimes C_2$ where
$A$ is a
non-trivial abelian group of odd order and the non-identity element of
$C_2$ acts on $A$ by inversion. 
\end{thmlist}
\end{theorem}

Here $A_n$ denotes the alternating group of degree $n$,
and $C_7 \rtimes C_3$ is the Frobenius group of order $21$,
which we may define
by $C_7 \rtimes C_3 = \left< g, h : g^7 = h^3 = 1,g^h = g^2
\right>$.
Recall that an element $g$ of a group $G$
is said to be \emph{self-centralising} if
$C_G(g)=\langle g\rangle$.
The infinite family of groups
of the form $A \rtimes C_2$ appearing in the theorem
 consists precisely of those
groups which contain a self-centralizing involution;
we include a brief proof of this fact as part of the proof of
Proposition~\ref{prop:scinvol}.
This proposition implies that if $G$ is such a group then
\hbox{$\kappa(G)=1/4+1/|G|-1/|G|^2$}. Thus one consequence of
Theorem~\ref{thm:uppergap} is that for all $\alpha>1/4$ there are only
finitely many groups~$G$ with $\kappa(G)\geq \alpha$. Since this is
not true when $\alpha\leq 1/4$, the threshold in
Theorem~\ref{thm:uppergap} is a natural one. 

The proof of Theorem~\ref{thm:uppergap}
uses the theory of Frobenius groups, and a
theorem of 
Feit and 
Thompson~\cite{FeitThompson} on groups
with a self-centralizing element of order $3$.

\subsubsection*{Background.}
Though we are not aware of any prior work concerning~$\kappa(G)$ for a
general finite group $G$, there are related quantities that have been
much studied. We give a brief overview as follows.  Let $\cp(G)$ be
the commuting probability for $G$, that is, the probability that a
pair of elements of $G$, chosen independently and uniformly at random,
commute.
Let $k(G)$ be the number of conjugacy class of $G$.  It is a
result of Erd\H{o}s and Tur\'an \cite{ErdosTuranIV} that the number of
pairs of commuting elements of $G$ is $|G|k(G)$.  From this it
follows that $\cp(G)=k(G)/|G|$.

In \cite[Lemma~2.4]{Lescot}, Lescot proved that if
$G$ and $H$ are
isoclinic finite groups, then $\mathrm{cp}(G)=\mathrm{cp}(H)$.
Our Theorem~\ref{thm:isoclinism} gives the analogous result
for the conjugacy probability $\kappa$.

In \cite{Gustafson} Gustafson showed that if $G$ is non-abelian, then
$\cp(G) \le 5/8$. This result is
sharp, since the upper bound is realized by the dihedral group of
order $8$, and in fact, by any group~$G$ such that $Z(G)$
has index~$4$ in $G$. Heuristically one might expect that if $G$ is a finite group for which
$\cp(G)$ is large then $\kappa(G)$ should be small, and \emph{vice
versa}.  Indeed, by Theorem~\ref{thm:lowergap}, the non-abelian groups
for which $\cp(G)$ is largest are precisely
those for which $\kappa(G)|G|$ is smallest.  However, this heuristic can mislead. For example if
$G$ is a group with a self-centralizing involution then,
by Proposition~\ref{prop:scinvol} below, $\kappa(G) \ge
1/4$ and $\cp(G) \ge 1/4$.

It is easier to make a connection between $\kappa(G)$ and $k(G)$.
It follows from \eqref{eq:kappa} that
$\kappa(G) \ge 1/k(G)$ with equality if and only if
$G$ is abelian. It was shown by Dixon \cite{Dixon1973} 
that if $G$ is
a non-abelian finite simple group then $k(G)/|G| \le 1/12$;
hence
$\kappa(G) \ge 12/|G|$ for all non-abelian finite simple groups $G$.

Various lower bounds for $k(G)$ have been given in terms of~$|G|$.
We refer the reader
to the survey article by Bertram \cite{Bertram}
which contains a wealth of information about the history and current
state of the problem. Of the general results known, the best
asymptotically is
\[
k(G) \ge \frac{c\log |G|}{(\log\log |G|)^7},
\]
due to Keller \cite{Keller}; it strengthens a result of Pyber
\cite{Pyber} by improving the exponent in the denominator from $8$ to
$7$.

We mentioned above that the proof of Theorem~\ref{thm:uppergap}
uses a theorem of 
Feit and 
Thompson~\cite{FeitThompson} on groups
with a self-centralizing element of order $3$.
This result was the precursor to a number of other
structural results on groups with a centralizer of small degree;
for example, in \cite[Corollary 4]{Herzog},
Herzog built on work of Suzuki
\cite{Suzuki} to classify
all finite simple groups with a centralizer of order at most~$4$.
It would have been possible to base our proof
of Theorem~\ref{thm:uppergap} on this classification.

We note that there has also been considerable
work on the general subject
of inferring structural information about groups from divisibility
properties of centralizer sizes.  We refer the reader to Camina and
Camina~\cite{CaminaCamina} for a survey of this part of the literature.

Finally, we mention two other `gap' results that are in a similar
spirit to Theorem~\ref{thm:uppergap}. 
C.T.C.~Wall~\cite{Wall} has shown that
the proportion of elements $x\in G$ satisfying $x^2=1$ is either~$1$
(in which case $G$ is an elementary abelian $2$-group) or it is at most
$3/4$. Laffey~\cite{Laffeyp,Laffey3} has proved that if
$G$ is a finite group then either
$G$ has exponent $3$, or the proportion of elements $x\in G$
satisfying $x^3=1$ is at most $7/9$. An important observation
in Laffey's proof is that if $N$ is a normal subgroup of $G$ then
the proportion of elements $xN \in G/N$ such that $(xN)^3 = 1$
is at least
as great as the corresponding
proportion in~$G$. We prove the analogous result
for $\kappa$ in Lemma~\ref{lemma:quo} below; in the final section
of this paper we outline a common framework for these results.

G.E.~Wall~\cite{GWall} has
constructed a $5$-group such that $24/25$ of its elements have order
dividing~$5$, so if an analogous result holds for the equation
$x^5=1$, the `gap' must be quite small.

\subsection{Results on symmetric groups}

Turning to the particular case of the symmetric groups $S_n$, we
provide a uniform bound on $\kSn$ in the following theorem.

\begin{theorem}
\label{thm:conjugate_uniform_bound}
For all positive integers~$n$ we have $\kSn \le C_\kappa /n^2$,
where $C_\kappa = 13^2\, \kappa(S_{13})$.
\end{theorem}
\noindent

It is clear that the bound in Theorem~\ref{thm:conjugate_uniform_bound}
is achieved when $n=13$, and so the constant $C_\kappa$ is the best
possible.
(Calculation shows that $C_\kappa \approx 5.48355$; for
the exact value see~Lemma~\ref{lemma:kappa_comp}.)
Theorem~\ref{thm:conjugate_uniform_bound} is
proved by induction on~$n$, using the inequality for~$\kSn$
established in Proposition~\ref{prop:conjugate_recurrence}.
An interesting feature of our argument is that
to make the induction go through, we require
the exact values of $\kSn$ for $n \le 80$: the
Haskell \cite{PeytonJones} source code
used to compute these values is available from the third
author's website: \url{http://www.ma.rhul.ac.uk/~uvah099/}.
Where feasible we have also verified these values using
Mathematica \cite{Mathematica}.

Using Theorem~\ref{thm:conjugate_uniform_bound} we are able to give an
elementary proof of the following asymptotic result, first proved by
Flajolet, Fusy, Gourdon, Panario and Pouyanne~\cite[\S 4.2]{FlajoletEtAl} using
methods from analytic combinatorics.

\begin{theorem}\label{thm:Plimit}
Let $A_\kappa = \sum_{n=1}^\infty \kSn$. Then $\kSn \sim
A_\kappa/n^2$ as \hbox{$n \rightarrow \infty$}.
\end{theorem}
\noindent That $A_\kappa$ is well defined follows from Theorem~\ref{thm:conjugate_uniform_bound}.
Flajolet \emph{et al} give the value of
$A_\kappa$ to $15$ decimal places as
$4.26340\,35141\,52669$.

We denote by $\rho(S_n)$ the probability
that if two
elements of~$S_n$ are chosen independently and uniformly at
random, then they have conjugates that commute.
The methods
we use to prove
Theorems~\ref{thm:conjugate_uniform_bound} and~\ref{thm:Plimit}
can be adapted to prove the analogous results for $\rho(S_n)$.

A useful general setting for this probability is given by the
relation on a group $G$, defined by $g\sim h$ if $g$ commutes 
with a conjugate of $h$. This relation
naturally induces a relation on the conjugacy classes of $G$: 
classes~$C$ and $D$ are said to \emph{commute} if they contain 
elements that commute.
Some aspects of this relation
have been described by the second and third authors in \cite{BritnellWildon} and \cite{BritnellWildonGL}; the latter paper
describes commuting classes in the case where $G$ is a general linear group.

We prove the following analogue of Theorem~\ref{thm:conjugate_uniform_bound}.

\begin{theorem}
\label{thm:conj_comm_uniform_bound}
For all positive integers~$n$ we have $\rho(S_n) \le C_\rho /n^2$,
where $C_\rho = 10^2 \rho(S_{10})$.
\end{theorem}
\noindent

Again it is clear that the bound in Theorem~\ref{thm:conj_comm_uniform_bound}
is achieved when $n=10$, and so the constant $C_\rho$ is the best
possible. (Calculation shows that $C_\rho \approx 11.42747$; for
the exact value see Lemma~\ref{lemma:rho_comp}.) To make
the induction go through we require the exact values of
$\rho(S_n)$ for $n \le 35$; these were found using the
necessary and sufficient condition given in \cite[Proposition~4]{BritnellWildon}
for two conjugacy classes of the symmetric group to commute,
and the software already mentioned.

Our asymptotic result on $\rho(S_n)$ is as follows.

\begin{theorem}\label{thm:comm_conj_Plimit}
Let $A_\rho =
\sum_{n=1}^\infty \rho(S_n)$. Then $\rho(S_n) \sim
A_\rho/n^2$ as $n \rightarrow \infty$.
\end{theorem}

\noindent
That $A_\rho$ is
well defined follows from Theorem~\ref{thm:conj_comm_uniform_bound}.
It follows from
Theorem~\ref{thm:conj_comm_uniform_bound} and the
exact value of $\sum_{m=0}^{30}\rho(S_m)$ given
in Lemma~\ref{lemma:rho_comp} that $6.1
<A_\rho < 6.5$.

The striking
similarity in the asymptotic behaviour of $\kSn$ and
$\rho(S_n)$ may be seen in the following corollary of
Theorems~\ref{thm:Plimit} and~\ref{thm:comm_conj_Plimit}
together with
the numerical estimates for $A_\kappa$ and~$A_\rho$ stated after these 
theorems.

\begin{corollary}
Let $\sigma$ and $\tau \in S_n$ be chosen independently and
uniformly at random.
If $\sigma$ and $\tau$ have conjugates
that commute, then,
provided~$n$ is sufficiently large, the probability that $\sigma$ and $\tau$ are conjugate is at least $42 / 65$.

\end{corollary}

\subsubsection*{Background.}
The paper \cite{FlajoletEtAl} by Flajolet \emph{et al} contains
the only prior work in this area of which the authors are aware.
Besides their proof
of Theorem~\ref{thm:Plimit}, they also
show in their Proposition~4 that
$\kSn =
A_\kappa/n^2 + \mathrm{O}\bigl( (\log n)/n^3 \bigr)$. This result
of course implies our Theorem~\ref{thm:conjugate_uniform_bound}.
Their proof does not, however, lead to explicit
bounds on $\kSn$; nor can their methods
be applied to $\rho(S_n)$. It therefore appears to be difficult
to give a more precise estimate for the constant~$A_\rho$ in
Theorem~\ref{thm:comm_conj_Plimit}. The integer
sequence $n!^2 \kappa(S_n)$ is A087132 in the
On-line Encyclopedia of Integer Sequences \cite{OEIS};
the sequence $n!^2 \rho(S_n)$ now appears as A192983.

The probabilities measured by $\kSn$ and $\rho(S_n)$ depend
only on the cycle types of $\sigma$ and $\tau$, and so these theorems may be regarded
as statements about the cycle statistics of a random permutation.
Rather than attempt to do justice to the enormous literature
on these cycle statistics, we shall merely recall some
earlier results relevant to our proofs.

It is critical to the success of our approach
that if $n$ is large
compared with $k$, then almost all permutations in $S_n$ contain
a cycle of length at least~$k$. The explicit bounds we require
are given in \S\ref{sec:short_cycles} below.
Somewhat weaker estimates
can be deduced from a
fundamental result, due to Goncharov \cite{Goncharov1},
which states that
if $X_i$ is the number of
$i$-cycles of a permutation in $S_n$ chosen uniformly
at random, then as $n$ tends to infinity, the
limiting distribution of the $X_i$ is as
independent Poisson
random variables, with $X_i$ having mean $1/i$.

It should also be noted that if
$L_n$ is the random variable whose value is the longest cycle of
a permutation $\sigma_n$
chosen uniformly at random from $S_n$, then,
as shown in \cite{Goncharov2}, the distribution of
$L_n/n$ tends to a limit.
The moments of the limiting distribution
of the $r$th longest cycle in~$\sigma_n$ were calculated by Shepp and Lloyd in
\cite{SheppLloyd}, who also showed that
the limit of  $\E (L_n/n)$ as $n\rightarrow \infty$
is approximately~$0.62433$.
The reader is referred to
Lemma~5.7 and (1.36)
in \cite{ArratiaEtAl} for an interesting formulation of these results
in terms of the Poisson--Dirichlet distribution.

\subsection{Structure of the paper}
\label{sec:outline}

The remainder of the paper is structured as follows. In \S\ref{sec:lowergap}
we prove Theorem~\ref{thm:lowergap}. In \S\ref{sec:uppergap_preliminaries}
we prove the preliminary lemmas needed for Theorem~\ref{thm:uppergap};
this theorem is then proved in \S\ref{sec:uppergap}. The isoclinism
result in Theorem~\ref{thm:isoclinism} is proved in \S\ref{sec:isoclinism}.

The second half of the paper on symmetric groups begins
with \S\ref{sec:short_cycles} where we give the bounds we require
on the probability that a permutation of $S_n$, chosen uniformly
at random, has only cycles of length strictly less than a fixed length $k$.
In \S\ref{sec:ineq} we establish a recursive bound on~$\kSn$ that is
critical to our approach;
in \S\ref{sec:kappa_upper_bound} we use this bound to prove
Theorem~\ref{thm:conjugate_uniform_bound}.
The asymptotic result on $\kSn$
in Theorem~\ref{thm:Plimit} is proved in \S\ref{sec:kappa_limit}.
To prove the analogous results on $\rho(S_n)$ stated in
 Theorems~\ref{thm:conj_comm_uniform_bound}
and~\ref{thm:comm_conj_Plimit}
we need some
further bounds, which we collect in~\S\ref{sec:comm_conj_bounds_for_permutations}.
The proofs of these theorems are then given in~\S\ref{sec:rho_proofs}.

Some final remarks and open problems
are presented in~\S\ref{sec:final}.

\section{Proof of Theorem~\ref{thm:lowergap}}
\label{sec:lowergap}

Let $G$ be a finite non-abelian group.
The contribution of the
central elements of $G$ to the sum in~\eqref{eq:kappa} is
$|Z(G)|/|G|^2$. If the non-central classes have sizes $a_1$, \ldots,
$a_\ell$ then they contribute $(a_1^2 + \cdots +
a_\ell^2)/|G|^2$.
Ignoring any divisibility restrictions on the~$a_i$
for the moment, we see that since $(b+c)^2 > b^2 + c^2$ for all $b$,
$c \in \N$, the sum $a_1^2 + \cdots + a_\ell^2$
 takes its minimum value when $a_i \le 3$ for all
$i$. Moreover, since $3^2 + 3^2 > 2^2 + 2^2 + 2^2$, we have
$a_i=3$ for at most one $i$ at this minimum value. Therefore
$a_1^2 + \cdots + a_\ell^2 \ge 2^2(|G| - |Z(G)|)/2$, with
equality exactly when $a_i=2$ for all $i$.
It
follows that
\[ \kappa(G) \ge \frac{|Z(G)|}{|G|^2} +
\frac{2^2}{|G|^2} \Bigl( \frac{|G|-|Z(G)|}{2}\Bigr) 
 = \frac{1}{|G|} \Bigl( 2 - \frac{|Z(G)|}{|G|} \Bigr) .\]
Since $G/Z(G)$ is non-cyclic, we have $|Z(G)| \le |G|/4$. Hence
$\kappa(G) \ge 7/4|G|$, and equality holds if and only if $Z(G)$
has index $4$ in $G$ and every non-central conjugacy class
has size $2$. However if $G$ is any group such
that $Z(G)$ has index $4$ in $G$
then, for any $g \in G \backslash Z(G)$, the
subgroup $\Cent_G(g) = \left<Z(G), g\right>$ has index $2$ in $G$.
Hence the condition that $Z(G)$ has index $4$ in $G$
is both necessary and sufficient.
This completes the proof of Theorem~\ref{thm:lowergap}.

\section{Preliminaries for the proof of Theorem~\ref{thm:uppergap}}
\label{sec:uppergap_preliminaries}
We begin the proof of Theorem~\ref{thm:uppergap} by collecting the three
preliminary lemmas we need. Of these, Lemma~\ref{lemma:quo} immediately
below is key. We give a more general version of this lemma
in~\S\ref{sec:final} at the end of the paper.

\begin{lemma}\label{lemma:quo}
Let $G$ be a finite group. If $N$
is a non-trivial normal subgroup of $G$ then
$\kappa(G) < \kappa(G/N)$.
\end{lemma}

\begin{proof}
Writing $\sim$ for the conjugacy relation, we have
\[ \{(g,h) \in G \times G : g \sim h \} \subset
   \{(g,h) \in G \times G : gN \sim hN \}.
\]
The inclusion is strict, because if $g$ is a non-identity
element of~$N$ then $(g,1)$ lies in the right-hand set but not the
left-hand set.  The proportion of pairs $(g,h) \in G \times G$ lying
in the smaller set is $\kappa(G)$. When $g$, $h \in G$ are chosen
independently and uniformly at random, $gN$ and $hN$ are
independently and uniformly distributed across the elements of $G/N$.
So the proportion of pairs $(g,h) \in G \times G$ lying
in the larger set is $\kappa(G/N)$. Hence $\kappa(G) < \kappa(G/N)$.\quad
\end{proof}

We shall also need a straightforward result on
the conjugacy probability in Frobenius groups.
Recall that a transitive
permutation group~$G$ acting faithfully on a finite set $\Omega$
is said to be a \emph{Frobenius group} if each non-identity element of $G$
has at most one fixed point. It is well known (see for example
\cite[Ch.~4, Theorem~5.1]{Gorenstein}) that if $G$ is a Frobenius group with point
stabiliser~$H$ then $G$ has a regular
normal subgroup~$K$ such that $G = KH$.
The subgroup $K$ is known as the \emph{Frobenius kernel} of~$G$.
Any non-identity
element of $H$ acts without fixed points on $K$, and so by
a famous theorem of Thompson
(see \cite[Theorem~1]{Thompson} or \cite[Ch.~10, Theorem~2.1]{Gorenstein}),
$K$ is nilpotent.

\begin{lemma}\label{lemma:frobenius}
If $G$ is a Frobenius group with point stabiliser $H$ and Frobenius
kernel~$K$ then

\[ \kappa(G) = \frac{1}{|G|^2} + \frac{1}{|H|} \Bigl(
\kappa(K) - \frac{1}{|K|^2} \Bigr) +
\Bigl( \kappa(H) - \frac{1}{|H|^2} \Bigr). \]
\end{lemma}

\begin{proof}
It follows from standard properties of Frobenius groups, see for
example \cite[Lemma~7.3]{Isaacs} that
\begin{equation}
\label{eq:decomp}
 G = K \cup \bigcup_{g \in K} (H^g \backslash \{ 1\})
\end{equation}
where the union is disjoint.
Every conjugacy class of $G$
is either contained in $K$, or has a representative in $H$.
Let $g_1, \ldots, g_r \in K$ be representatives
for the non-identity conjugacy classes contained in $K$,
and let $h_1, \ldots, h_s \in H$ be
representatives for the remaining non-identity conjugacy
classes of $G$.
By~\eqref{eq:kappa} in \S1 we have
\begin{equation}
\label{eq:frob_kappa_proof}
\kappa(G) = \frac{1}{|G|^2} + \sum_{i=1}^r
\frac{1}{|\Cent_G(g_i)|^2}
+ 
\sum_{i=1}^s \frac{1}{|\Cent_G(h_i)|^2}.
\end{equation}

Since $H$ acts without fixed points on $K$,
no two non-identity elements of $H$ and $K$ can commute.
Hence $\Cent_G(g) = \Cent_K(g)$ for
each $g \in K$. Therefore,
when the conjugacy action is restricted to~$K$,
each $g_i^G$ splits into $|H|$ disjoint conjugacy classes of $K$.
Moreover, \eqref{eq:decomp} shows that
any pair of conjugates of $H$ by distinct elements of $K$
intersect only in the identity.
Hence  $\Cent_G(h) = \Cent_H(h)$ for each $h \in H$
and the $G$-conjugacy classes in $G \setminus K$ are
in bijection with the $H$-conjugacy classes in $H$.
We therefore have
\begin{align*}
\kappa(K) &= \frac{1}{|K|^2} + |H|\sum_{i=1}^r \frac{1}{|\Cent_G(g_i)|^2}, \\
\kappa(H) &= \frac{1}{|H|^2} + \sum_{i=1}^s \frac{1}{|\Cent_G(h_i)|^2}.
\end{align*}
The lemma now follows from \eqref{eq:frob_kappa_proof}
using these equations.
\end{proof}

To state our final preliminary lemma we
shall need the \emph{majorization} (or dominance)
order, denoted $\succeq$, which is defined on $\R^k$ by setting
\[ (x_1, x_2, \ldots, x_k) \succeq (y_1, y_2, \ldots, y_k) \]
if and only if
$\sum_{i=1}^j x_i \ge \sum_{i=1}^j y_i$
for all $j$ such that $1 \le j \le k$.

\begin{lemma}\label{lemma:ineq}
Let $x$, $y \in \R^k$ be
decreasing $k$-tuples of real numbers such that
$\sum_{i=1}^k x_i
= \sum_{i=1}^k y_i = 1$. Suppose that $x \succeq y$.
Then
$\sum_{i=1}^k x_i^2 \ge \sum_{i=1}^\ell y_i^2$,
and equality holds if and only if $x=y$.
\end{lemma}

\begin{proof}
This follows from Karamata's inequality
(see \cite[page 148]{Karamata}) for the function
\hbox{$f(x) = x^2$}.
\end{proof}

For notational convenience, we may suppress a sequence
of final zeros when writing elements of $\R^k$.

\bigskip
\section{Proof of Theorem~\ref{thm:uppergap}}
\label{sec:uppergap}

We prove Theorem~\ref{thm:uppergap} by using Lemma~\ref{lemma:ineq} to
give a fairly strong restriction on the centralizer sizes in a finite
group $G$ such that $\kappa(G) \ge 1/4$. The groups falling into each case
in Proposition~\ref{prop:class} below
are then classified using Lemmas~\ref{lemma:quo}
and~\ref{lemma:frobenius}.  From case (i) we get the infinite family
in Theorem~\ref{thm:uppergap}, and from case (ii) we get $A_4$ and
$C_3 \rtimes C_7$.  The unique groups in cases~(iii) and~(iv)
are $S_4$ and $A_5$, respectively.

We shall denote by
$c_i(G)$ the size of the $i$th smallest centralizer in a finite group~$G$.

\begin{proposition}\label{prop:class}
If $G$ is a finite group such that $\kappa(G) \ge 1/4$ then
either~$|G| \le 4$, or one of the following holds:
\begin{thmlist}
\item $c_1(G) = 2$;
\item $c_1(G) = c_2(G) = 3$;
\item $c_1(G) = 3$ and $c_2(G) = c_3(G) = 4$;
\item $c_1(G) = 3$, $c_2(G) = 4$ and $c_3(G) = c_4(G) = 5$.
\end{thmlist}
\end{proposition}

\begin{proof}
Let $k$ be the number of conjugacy classes of $G$ and
let 
\[ \r(G) = (1/c_1(G), 1/c_2(G), \ldots, 1/c_{k}(G)). \]
Since the size of the $i$th largest conjugacy
class of $G$ is $|G|/c_i(G)$, we have $\sum_{i=1}^k 1/c_i(G) = 1$.

If $c_1(G) > 3$ then, using the notational convention
established at the end of \S\ref{sec:uppergap_preliminaries}, we have
$(\frac{1}{4}, \frac{1}{4}, \frac{1}{4}, \frac{1}{4})
\succeq \r(G)$.
Hence either $|G| = 4$ or, by Lemma~\ref{lemma:ineq} and~\eqref{eq:kappa}
we have
\[ \kappa(G) = \sum_{i=1}^k \frac{1}{c_i(G)^2} < 1/4. \]

If $c_1(G)=2$, then $G$ lies in case~(i). We may therefore assume that
$c_1(G) = 3$. The remainder of the argument proceeds along similar
lines: if $c_2(G) > 4$ then $(\frac{1}{3}, \frac{1}{5}, \frac{1}{5},
\frac{1}{5}, \frac{1}{15}) \succeq \r(G)$, and by Lemma 9 we have
$\kappa(G) \le 53/255 < 1/4$.  Hence either $G$ lies in case~(ii)
or \hbox{$c_2(G) = 4$}. Assume that $c_2(G)=4$. If $c_3(G) >
5$ then $(\frac{1}{3}, \frac{1}{4}, \frac{1}{6}, \frac{1}{6},
\frac{1}{12}) \succeq \r(G)$ and we have $\kappa(G) \le 17/72 <
1/4$. Therefore either $c_3(G) = 4$, and $G$ lies in case (iii), or
\hbox{$c_3(G) = 5$}.  In this final case, if $c_4(G) > 5$ then
$(\frac{1}{3}, \frac{1}{4}, \frac{1}{5}, \frac{1}{6}, \frac{1}{20})
\succeq \r(G)$ and we have $\kappa(G) \le 439/1800 < 1/4$. Therefore
$c_4(G) = 5$, and $G$ lies in case~(iv).
\end{proof}

\subsection{Self-centralizing involutions}
If $G$ lies in the first case of Proposition~\ref{prop:class} then
$G$ contains an element $t$ such that $|C_G(t)| = 2$. Since $\left<t\right>
\le C_G(t)$ we see that $t$ is a self-centralizing involution.

\begin{proposition}\label{prop:scinvol}
Suppose
that $G$ is a finite group and that $t \in G$
is a self-centralizing involution.
Then $\kappa(G) = \frac{1}{4} + \frac{1}{|G|}  -\frac{1}{|G|^2}$
and $G$ has a normal abelian subgroup $A$ of odd order
such that $|G : A| = 2$ and $t$ acts on $A$ by mapping
each element of $A$ to its inverse.
\end{proposition}

\begin{proof}
The action of $G$ on the conjugacy class $t^G$
makes $G$ into a Frobenius group. Let~$A$ be the Frobenius kernel.
Since $t$ acts without fixed points on~$A$, we see that~$A$
has odd order. Now, proceeding as in \cite[Ch.~10, Lemma~1.1]{Gorenstein},
we observe the map on~$A$ defined by $g \mapsto g^{-1}g^t$
is injective. Therefore every element in $A$ is of the form
$g^{-1}g^t$ for some~$g \in A$, and hence~$t$ acts on $A$
by mapping each element to its inverse. It is an easy exercise
to show this implies that $A$ is abelian.

We have therefore shown that $G$ is a Frobenius group
with abelian kernel $A$ and complement isomorphic to $C_2$.
Since $\kappa(A) = 1/|A|$ it
follows from Lemma~\ref{lemma:frobenius} that
\[ \kappa(G) 
= \frac{1}{|G|^2} + \frac{1}{2} \Bigl( \frac{1}{|A|}
- \frac{1}{|A|^2}\Bigr) + \Bigl( \frac{1}{2} - \frac{1}{4} \Bigr) \\
= \frac{1}{4} + \frac{1}{|G|} - \frac{1}{|G|^2} .\]
This completes the proof of the proposition.
\end{proof}

\subsection{Groups with a self-centralizing three-element}
To deal with the remaining three cases of Proposition~\ref{prop:class}
we shall need the following straightforward lemmas.

\begin{lemma}\label{lemma:fpf}
Let $K$ be a finite group and let $x : K \rightarrow K$ be
a fixed-point-free automorphism of prime order $p$.
Then $|K| \equiv 1$ mod $p$ and if
$L$ is an $x$-invariant subgroup of $K$ then
the induced
action of $x$ on $K/L$ is fixed-point-free.
\end{lemma}

\begin{proof}
The identity is the unique fixed point of $x$ on $K$ and
so $K \backslash \{1\}$ is a union of orbits of $x$.
Hence $|K| \equiv 1$ mod $p$.
Suppose that $Lg$ is a proper coset
of $L$ and that $Lg$ is fixed by~$x$. Then $Lg$ is a union
of orbits of~$x$ and so $|Lg|$ is divisible by $p$,
a contradiction.
\end{proof}

\begin{lemma}\label{lemma:pcentre}
Let $G$ be a finite group containing a self-centralizing element
$x$ of order $3$. Then $\left<x \right>$ is a
self-centralizing Sylow $3$-subgroup
of~$G$.
\end{lemma}

\begin{proof}
Let $P$ be a Sylow $3$-subgroup of $G$ containing $x$. Since
$x$ is centralized by $Z(P)$ we see that $x \in Z(P)$. But then $P\leq\Cent_G(\langle x \rangle)=\langle x \rangle$. Hence $\left<x\right> = P$.
\end{proof}


\begin{lemma}\label{lemma:frob3p}
Suppose that $K$ is a non-trivial finite group and $x : K \rightarrow K$
is a fixed-point-free automorphism of order $3$. 
If $|K / K'| \le 12$
then either $K \cong C_2 \times C_2$ or $K \cong C_7$.

\end{lemma}

\begin{proof}
Suppose first of all that $K$ is abelian. Then, from Lemma~\ref{lemma:fpf}
we see that $|K| = 4$, $7$ or $10$. 
The latter case is  impossible as $C_2 \times C_5$ does not admit 
an automorphism of order $3$.
Similarly $C_4$ is ruled out, leaving the groups in the lemma.

Now suppose that $K$ is not abelian.  By the theorem of Thompson
mentioned before Lemma~\ref{lemma:frobenius},~$K$ is nilpotent. The
derived group $K'$ is a characteristic subgroup of $K$, so $x$ also
acts on the non-trivial quotient $K/K'$. It follows from
Lemma~\ref{lemma:fpf} that the action of $x$ on $K/K'$ is
fixed-point-free.  Hence, by the previous paragraph, we either have
$K/K' \cong C_2 \times C_2$, or $K/K' \cong C_7$. A nilpotent group
with a cyclic abelianisation is cyclic, so the latter case is
impossible.

In the former case, let $K/K' = \left<sK', tK'\right>$ and let
$L = [K,K']$ be the second term in the lower central series of $K$.
Note that $L$ is a $\left<x\right>$-invariant subgroup of $K$, so, by
Lemma~\ref{lemma:fpf}, $x$ acts without fixed points on $K/L$,
and also on $(K/L)'$.
It is clear that $(K/L)'$ is generated by $[sL,tL]$,
and so it is a non-trivial cyclic $2$-group.
But no such group admits a fixed-point-free automorphism of order $3$,
a contradiction. 
\end{proof}

\subsubsection*{Case (ii) of Proposition~\ref{prop:class}.}
In this case we have $c_1(G) = c_2(G) = 3$ and so there are two
distinct conjugacy class of self-centralizing elements of order
$3$. Let $x$ be contained in one of these classes. By Lemma~\ref{lemma:pcentre},
 each element
of order $3$ is contained in some conjugate of~$\left<x\right>$,
and so we see
that $x$ is not conjugate to $x^{-1}$. The subgroup~$\left<x\right>$ 
is therefore not only self-centralizing, but also
self-normalizing. If $\left<x\right>$ is the unique Sylow $3$-subgroup
of $G$
then $G \cong C_3$. Otherwise it follows at once that the action of
$G$ on the conjugates of $\left<x\right>$ makes $G$ into
a Frobenius group with complement~$\left<x\right>$. Let $K$ be the
Frobenius kernel of $G$.

By Lemma~\ref{lemma:fpf}, $x$ acts on $K/K'$ as a fixed-point-free
automorphism. Hence $G/ K'$ is a Frobenius group.  Applying
Lemmas~\ref{lemma:quo} and~\ref{lemma:frobenius} to $G/K'$, and using
the fact that $\kappa(K/K') = 1/|K/K'|$, we get
\[
\kappa(G) \leq \kappa(G/K') =\frac{1}{3^2|K/K'|^2} + \frac{1}{3} \Bigl( \frac{1}{|K/K'|} -
\frac{1}{|K/K'|^2} \Bigr) + \Bigl( \frac{1}{3} - \frac{1}{9} \Bigr).
\]
If $\kappa(G) \ge 1/4$, the previous inequality implies that
\[ \frac{1}{36} \le \frac{-2}{9|K/K'|^2}  + \frac{1}{3|K/K'|} \]
and it follows that $| K / K'| \le 11$. Since $K$ is nilpotent, $K/K'$
is non-trivial.
By Lemma~\ref{lemma:frob3p} we get that 
either $K \cong C_2 \times C_2$, in which case 
 $G\cong(C_2 \times C_2) \rtimes C_3\cong A_4$,
or $K \cong C_7$, in which case $G\cong C_7 \rtimes C_3$.

\subsubsection*{Cases (iii) and (iv) of Proposition~\ref{prop:class}.} Lemma~\ref{lemma:pcentre} shows that the Sylow $3$-groups of $G$ are self-centralising and of order $3$.
By hypothesis, there is a unique conjugacy class of elements with centraliser of order~$3$, and so we see that all $3$-elements of $G$ are conjugate and are self-centralising.

We shall need the main theorem
in~\cite{FeitThompson}. We repeat the statement of this theorem below.

\begin{utheorem}[\hbox{[W. Feit and 
J.~G. Thompson \cite{FeitThompson}]}]
Let $G$ be a finite group which contains a self-centralizing subgroup
of order $3$. Then one of the following statements is true.

\begin{thmlist}
\item[\emph{(I)}] $G$ contains a nilpotent normal subgroup $N$ such that $G/N$
is isomorphic to either $A_3$ or~$S_3$.

\item[\emph{(II)}]  $G$ contains a normal subgroup $N$ which is a $2$-group such
that $G/N$ is isomorphic to~$A_5$.

\item[\emph{(III)}] $G$ is isomorphic to $\mathrm{PSL}(2,7)$.
\end{thmlist}
%
%
\end{utheorem}

The conjugacy classes of $\mathrm{PSL}(2,7)$ have centralizer sizes
$3$, $4$, $7$, $7$, $8$ and $168$, and so $\mathrm{PSL}(2,7)$
does not fall into any of our cases.  We may therefore ignore the
third possibility in the Feit--Thompson theorem.  (Explicit
calculation shows that in fact $\kappa(\mathrm{PSL}(2,7)) = 3247/14112$.)  Moreover, groups with a quotient isomorphic to $A_3$
contain at least two conjugacy classes of $3$-elements, and so
$G/N\cong S_3$ whenever possibility~(I) occurs.

Suppose that our Case~(iii) holds. Half of the elements of $G$ lie in the two
conjugacy classes with a centraliser of order $4$, and in particular
at least half the elements of $G$ are non-trivial $2$-elements. This
means that possibility~(II) of the Feit--Thompson theorem cannot
occur: the number of $2$-elements in $A_5$ is $16$, and so the
proportion of $2$-elements in a group having~$A_5$ as a quotient
cannot exceed $16/60$. We may
therefore assume that $G$ has a nilpotent
normal subgroup $N$ such that~$G/N$ is isomorphic to $S_3$. Note that $N$ cannot be trivial, for then~\hbox{$c_1(G)=2$}.

Let $H$ be the subgroup of index $2$ in $G$ containing $N$.  The two
conjugacy classes with centralizer of size $4$ must form $G \backslash
H$, and if $x\in G$ has centralizer size $3$ then $x^G \subseteq H$.
The class $x^G$ splits into two when the conjugacy action is
restricted to $H$, and so, as in Case~(ii), we see that~$H$ is a
Frobenius group. Let~$K$ be its kernel, and so $H = K \rtimes \left<x
\right>$ where $x$ acts fixed-point-freely. Note that~$K$ is nilpotent, 
and since $N$ is non-trivial, $K$ is also non-trivial.  
The conjugacy classes of $G$ not contained in~$K$ contribute
$\frac{1}{4^2} + \frac{1}{4^2} + \frac{1}{3^2}$ to $\kappa(G)$, and so
we have $\kappa(G) = \frac{17}{72} + c$ where $c$ is the contribution
from the $G$-conjugacy classes contained in~$K$.  Let $k \in K$ be a
non-identity element.  Either $\Cent_G(k)$ meets $G \backslash H$, in
which case $k^G$ splits into $3$ classes, each with centralizer size
$|\Cent_G(k)|/2$, when the conjugacy action is restricted to~$K$, or
$\Cent_G(k) \le K$, in which case $k^G$ splits into $6$ classes, each
with centralizer size $|\Cent_G(k)|$.  Suppose that the former
conjugacy classes of $G$ have $G$-centralizer sizes $x_1$, \ldots,
$x_r$, and the latter have $G$-centralizer sizes $y_1$, \ldots, $y_s$.
By~\eqref{eq:kappa} we have
\[ \kappa(G) =
\frac{1}{|G|^2} +
\frac{17}{72}  + \sum_{i=1}^r \frac{1}{x_i^2}
+ \sum_{j=1}^s \frac{1}{y_j^2} \]
and
\[
\kappa(K) = \frac{1}{|K|^2} + \sum_{i=1}^r \frac{3}{(x_i/2)^2}
+ \sum_{j=1}^s \frac{6}{y_i^2}.
\]
Hence
\[
\kappa(G) \le \frac{17}{72} + \frac{\kappa(K)}{6}.
\]
Thus, by Lemma~\ref{lemma:quo},
\[
\kappa(G) \leq \frac{17}{72} +
\frac{\kappa(K/K')}{6} =\frac{17}{72} + \frac{1}{6|K/K'|}.
\]

Suppose that $\kappa(G) \ge 1/4$. The above inequality implies that
$|K/K'|\leq 12$ and hence, by Lemma~\ref{lemma:frob3p}, $K\cong
C_2\times C_2$ or $K\cong C_7$. The latter case cannot occur because then 
$G$ has order~$42$ and
so cannot contain a centralizer of order~$4$.
Hence we must have $K \cong C_2 \times C_2$ and $H = A_4$, and so
$G = S_4$. Therefore $S_4$ is the unique group lying in case (iii).

%

Finally, suppose Case~(iv) holds. 
Possibility (I) of the Feit--Thompson theorem cannot occur. To see
this, let $N$ be a normal subgroup of $G$ such that $G/N\cong S_3$,
and let $H$ be the subgroup of index $2$ in $G$ containing $N$. There
are conjugacy classes of $G$ with sizes $|G|/3$, $|G|/4$, $|G|/5$ and
$|G|/5$. All other conjugacy classes of $G$ have total size $|G|/60$.
Either $H$ or~$G \backslash H$ must contain the class of size $|G|/3$,
but no union of conjugacy classes of $G$ has size $|G|/2 - |G|/3 =
|G|/6$. So possibility~(II) must occur: there exists a
normal $2$-subgroup $N$ of $G$ such that $G/N\cong A_5$.

Let $g\in G$ lie in the unique conjugacy class of $G$ with a
centraliser of order $4$. Since any power of $g$ lies in $\Cent_G(g)$,
we see that $g$ is a $2$-element. Let $P$ be a Sylow $2$-subgroup
containing~$g$. Note that $N\leq P$, and $|P|=4|N|$. Since
$N_G(P)/N=N_{A_5}(P/N)$, and there is a $3$-element in the normaliser
of a Sylow $2$-subgroup of $A_5$, there exists an element $x\in
N_G(P)$ of order $3$. Since $\langle x\rangle$ is a self-centralising
Sylow $3$-subgroup, the action of $\langle x \rangle$ on $P$ is
fixed-point-free. In particular, since the action preserves $Z(P)$ we
see that $|Z(P)|\not=2$. Since $\langle g,Z(P)\rangle\subseteq
\Cent_G(g)$ and $|\Cent_G(g)|=4$, we must therefore have $|Z(P)|=4$
and $g\in Z(P)$. But then $P\leq \Cent_G(g)$ and so $|P|=4$. This
implies that $|N|=1$ and so $G\cong A_5$, as required.

\section{Proof of Theorem~\ref{thm:isoclinism}}
\label{sec:isoclinism}

We begin by reminding the reader of the definition of isoclinism.
Given a group $G$, we shall write $\overline{G}$ for the
quotient group $G/Z(G)$ and $\overline{g}$ for the coset $gZ(G)$
containing $g \in G$.
We define a map
$f_G:\overline{G}\times\overline{G} \longrightarrow G'$ by
$f_G(g_1Z(G),g_2Z(G))=[g_1,g_2]$, for
$g_1$, $g_2 \in G$. (It is clear that this map
is well-defined.)
Two finite groups $G$ and $H$ are said to be \emph{isoclinic}
if there exist isomorphisms
$\alpha:\overline{G}\longrightarrow \overline{H}$ and
$\beta:G'\longrightarrow H'$, such that for all $g_1,g_2\in G$,
\[
f_G(g_1,g_2)^\beta = f_H(g_1^\alpha,g_2^\alpha).
\]

Let $G$ and $H$ be isoclinic groups and let
$\alpha$, $\beta$, $f_G$ and $f_H$ be as in the definition
of isoclinism.
Since
\[ |g^G| = \bigl| [g,G] \bigr| = |f_G(\overline{g}, \overline{G})| \]
for each $g \in G$,
we have
\[ |G|\kappa(G)  
= \frac{1}{|G|}\sum_{g\in G}\bigl| [g,G] \bigr|  =
\frac{|Z(G)|}{|G|}\sum_{x\in \overline{G}}|f_G(x,\overline{G})|.
\]
Similarly,
\[
|H|\kappa(H)=\frac{|Z(H)|}{|H|}\sum_{y\in \overline{H}}|f_H(y,\overline{H})|.
\]
Now using the bijections $\alpha$ and $\beta$ we get
\[
\sum_{x \in \overline{G}} |f_G(x,\overline{G})| =
\sum_{x \in\overline{G}}
|f_G(x,\overline{G})^\beta|
= \sum_{x\in \overline{G}}|f_H(x^\alpha,\overline{H})| =
\sum_{y \in \overline{H}}|f_H(y,\overline{H})|.
\]
Since $\overline{G}$ and $\overline{H}$ are isomorphic, $|G|/|Z(G)| = |H|/|Z(H)|$ and so
\[
|G|\kappa(G)= |H|\kappa(H),
\]
as required by Theorem~\ref{thm:isoclinism}.

\section{Permutations that have only short cycles}
\label{sec:short_cycles}

We now turn to the proofs of our theorems on $\kSn$
and $\rho(S_n)$. The reader is referred to~\S\ref{sec:outline}
above for an outline of what follows.

For the remainder of the paper, we
let $\Omega_n = \{1,2, \ldots, n\}$. Unless we
indicate otherwise, we regard the symmetric
group $S_n$ as the set of permutations of
$\Omega_n$.
Let $s_k(n)$ be the probability that a permutation of~$\Omega_n$,
chosen uniformly at random, has only cycles of length strictly less
than~$k$.

In this section we establish bounds on $s_k(n)$. We begin
with the following well known lemma.

\begin{lemma}\label{lemma:1cycle}
Let $n \in \mathbf{N}$ and let $1\le \ell \le n$. Let $X \subseteq \Omega_n$
be an $\ell$-set.
If $\sigma$ is chosen uniformly at random from $S_n$ 
then
\begin{thmlist}
\item 
the probability
that $\sigma$ acts as an $\ell$-cycle on $X$ is
$\frac{1}{\ell} \binom{n}{\ell}^{-1}$;

\item the expected number of $\ell$-cycles in $\sigma$ is $1/\ell$;
\item the
probability that $1$ is contained in an $\ell$-cycle of $\sigma$
is~$1/n$.
\end{thmlist}
\end{lemma}

\begin{proof}
There are exactly $(\ell-1)!$ cycles of length $\ell$ in a symmetric group
of degree $\ell$. So the probability in (i) is
$(\ell-1)!(n-\ell)!/n! = \frac{1}{\ell} \binom{n}{\ell}^{-1}$.
It
follows that the expected number of $\ell$-cycles in $\sigma$ is~$1/\ell$ and so (ii) is established.

Define a random variable $X_i$ to be equal to $1$ if $i$ is contained in an $\ell$-cycle, and
 $0$ otherwise. The expected value of $\sum_{i=1}^n X_i$ is exactly $\ell$ times the expected number of $\ell$-cycles in a permutation. From part~(ii),
 we  see that $\E(\sum_{i=1}^n X_i)=1$. But $\E(X_i)$ does not depend on~$i$, hence $\E(\sum_{i=1}^nX_i)=n\E(X_1)$ and so $\E(X_1)=1/n$. Since $\E(X_1)$ is equal to the probability that $1$ is contained in an $\ell$-cycle,
the last part of the lemma is established.
\end{proof}

\begin{proposition}\label{prop:small_cycles}
For all $n,k \in \mathbf{N}$ with $k\geq 2$ we have $s_k(n) \le 1/t!$ where
$t = \lfloor n/(k-1) \rfloor$.
\end{proposition}

\begin{proof}
Let $\sigma$ be chosen uniformly at random from $S_n$.
By conditioning on the length of the cycle containing $1$, and using
Lemma~\ref{lemma:1cycle}(iii), we obtain the recurrence
\begin{equation}\label{eq:1rec} s_k(n) = \frac{1}{n} \sum_{j=1}^{k-1} s_k(n-j)
\end{equation}
for $n \ge k-1$. A simple calculation shows that
\begin{align*} s_k(n) - s_k(n+1) &=
\frac{1}{n(n+1)} \sum_{j=1}^{k-1} s_k(n-j) + \frac{s_k(n-(k-1))}{n} - \frac{s_k(n)}{n+1}
\\
&\ge \frac{s_k(n-(k-1)) - s_k(n)}{n}.
\end{align*}
Hence, if
\[ s_k(n-(k-1))
\ge \cdots \ge s_k(n-1) \ge s_k(n),
\]
then $s_k(n) \ge s_k(n+1)$. It is clear
that $s_k(n) = 1$ if $n \le k-1$, and so it follows by induction on $n$
that
 $s_k(n)$ is a decreasing function of $n$.
Now, by \eqref{eq:1rec} we have
\begin{equation}
\label{eq:skn_rec} s_k(n) \le \frac{k-1}{n} s_k(n-(k-1))
\end{equation}
for all $n \ge k-1$. Let $n = t(k-1) + r$ where $0 \le r < k-1$.
Repeated application of~\eqref{eq:skn_rec} gives
\begin{align*} s_k(t(k-1)+r) &\le \frac{k-1}{t(k-1)+r}\,
\frac{k-1}{(t-1)(k-1)+r} \,
\ldots \,\frac{k-1}{(k-1) + r}s_k(r)
 \\ &= \frac{1}{t+r/(k-1)}\,
\frac{1}{(t-1)+r/(k-1)}\, \ldots \,\frac{1}{1 + r/(k-1)} \\ &\le
\frac{1}{t!} \end{align*} as required.
\end{proof}

Using the elementary inequality
\begin{equation}
\label{eq:factorial_estimate}
t! \ge \Bigl( \frac{t}{\mathrm{e}} \Bigr)^t \quad\text{for $t \ge 1$},
\end{equation}
we derive the following corollary.

\begin{corollary}\label{cor:small_cycles}
For all $n, k \in \mathbf{N}$ with $k \ge 2$ we have
\[ s_k(n) \le \Bigl( \frac{\e}{t} \Bigr)^{t} \]
where $t = \lfloor n/(k-1) \rfloor$. \hfill$\square$
\end{corollary}

The following proposition leads to bounds on $s_k(n)$ that
are asymptotically poorer than
Corollary~\ref{cor:small_cycles}, but are
stronger when $n$ is small; we  use these bounds
in the proofs of Theorem~\ref{thm:conjugate_uniform_bound}
and Theorem~\ref{thm:conj_comm_uniform_bound}.

\begin{proposition}\label{prop:nsk_decreasing}
Let $k \in \mathbf{N}$ be
such that $k \ge 2$. Suppose that there exists $n_0\in \mathbf{N}$ such that
\[
ns_k(n)\geq(n+1)s_k(n+1)\text{ for }n\in\{n_0,n_0+1,\ldots ,n_0+k-2\}.
\]
Then $ns_k(n)\geq(n+1)s_k(n+1)$ for all $n\geq n_0$.
\end{proposition}

\begin{proof}
Let $x\geq n_0+k-1$, and suppose, by way of an inductive hypothesis, that $ns_k(n)\geq(n+1)s_k(n+1)$ for $n\in\{n_0,n_0+1,\ldots ,x-1\}$. Then, by~\eqref{eq:1rec},
\begin{align*}
(x+1)s_k(x+1)&=\sum_{j=1}^{k-1}s_k(x+1-j)\\
&=\sum_{j=1}^{k-1}\frac{(x+1-j)s_k(x+1-j)}{x+1-j}\\
&\leq\sum_{j=1}^{k-1}\frac{(x-j)s_k(x-j)}{x+1-j}\\
&<\sum_{j=1}^{k-1}\frac{(x-j)s_k(x-j)}{x-j}\\
&=xs_k(x).
\end{align*}
Hence the proposition follows by induction.
\end{proof}
In fact, it is always the case that $ns_k(n)\geq(n+1)s_k(n+1)$ for $n\geq k-1$. The authors have a combinatorial proof of this fact by means of explicitly defined bijections; because of the length of this proof we prefer the 
simpler approach adopted above.

\section{An inequality on $\kSn$}
\label{sec:ineq}

The bounds on $\kappa(S_n)$ in the
following proposition are critical to the proofs
of Theorems~\ref{thm:conjugate_uniform_bound} and~\ref{thm:Plimit}.

\begin{proposition}
\label{prop:conjugate_recurrence}
For all $n \in \mathbf{N}$ we have
\[
\kSn \le s_k(n)^2 + \sum_{\ell=k}^{n} \frac{\kappa(S_{n-\ell})}{\ell^2}.
\]
Moreover, if $k$ is such that $n/2 < k \le n$, then
\[
\kSn \ge \sum_{\ell=k}^n \frac{\kappa(S_{n-\ell})}{\ell^2}.
\]
\end{proposition}

\begin{proof}
Let $\sigma$ and $\tau$ be permutations of $\Omega_n$ chosen
independently and uniformly
at random.  Let $X,Y\subseteq \Omega_n$ be $\ell$-sets,
and write $\overline{X}$ and $\overline{Y}$ for the
complements of $X$ and $Y$ in $\Omega_n$, respectively. Let $E(X,Y)$ be
the event that $\sigma$ acts as an $\ell$-cycle on $X$, $\tau$ acts as
an $\ell$-cycle on $Y$ and the restrictions
$\overline{\sigma}$
and
$\overline{\tau}$
of
$\sigma$ and $\tau$ to $\overline{X}$ and $\overline{Y}$ respectively
have the same cycle structure. By Lemma~\ref{lemma:1cycle}(i),
the probability that $\sigma$ acts as
an $\ell$-cycle on $X$ and $\tau$ acts as an $\ell$-cycle on $Y$ is
$\ell^{-2} \binom{n}{\ell}^{-2}$.
Given that $\sigma$ and $\tau$ act as $\ell$-cycles on $X$
and $Y$ respectively, the permutations $\overline{\sigma}$
and $\overline{\tau}$ are
independently and uniformly distributed over
the symmetric groups on $\overline{X}$ and
$\overline{Y}$ respectively.
Hence the
probability that $\overline{\sigma}$
and $\overline{\tau}$ have the same cycle structure is precisely
$\kappa(S_{n-\ell})$. Thus
\[
\P(E(X,Y))=\binom{n}{\ell}^{-2} \frac{\kappa(S_{n-\ell})}{\ell^2} .
\]

If $\sigma$ and $\tau$ are conjugate, then either $\sigma$ and $\tau$
both have
cycles all of length strictly
less than~$k$, or there exist sets $X$ and $Y$ of
cardinality $\ell \ge k$ on which $\sigma$ and $\tau$ act as $\ell$-cycles
and such
that the restrictions of $\sigma$ to~$\overline{X}$
and $\tau$ to $\overline{Y}$ have the
same cycle structure. Therefore
\begin{align*}
\kSn &\le
s_k(n)^2+\sum_{\ell=k}^n\sum_{|X|=\ell}\sum_{|Y|=\ell}\P(E(X,Y))\\
&=s_k(n)^2+\sum_{\ell=k}^n\sum_{|X|=\ell}\sum_{|Y|=\ell} 
\binom{n}{\ell}^{-2} \frac{\kappa(S_{n-\ell})}{\ell^2} \\
& = s_k(n)^2+\sum_{\ell=k}^n\frac{\kappa(S_{n-\ell})}{\ell^2}.
\end{align*}
This establishes the first inequality of the proposition.

When $k > n/2$ the events $E(X,Y)$ with $|X|=|Y|\geq k$ are disjoint,
since a permutation can contain at most one cycle of length greater
than $n/2$. Thus
\begin{align*}
\kSn &\geq
\sum_{\ell=k}^n\sum_{|X|=\ell}\sum_{|Y|=\ell}\P(E(X,Y))\\
&=\sum_{\ell=k}^n \frac{\kappa(S_{n-\ell})}{\ell^2},
\end{align*}
as required.
\end{proof}

\section{A uniform bound on $\kSn$}
\label{sec:kappa_upper_bound}

In this section, we prove Theorem~\ref{thm:conjugate_uniform_bound}.
We shall require the following lemma.

\newcommand{\betak}{k}
\begin{lemma}\label{lemma:pfraclemma}
Let $n \in \mathbf{N}$ and let $0 < \betak < n/2$. Then 
\[ \sum_{\ell = \lceil n/2\rceil}^{n-\betak-1} \frac{1}{\ell^2(n-\ell)^2}
\le  \frac{1}{n^2\betak} + \frac{2 \log (n/\betak)}{n^3}.
\]
\end{lemma}

\begin{proof}
We have
\begin{align*}
\sum_{\ell = \lceil n/2\rceil}^{n-\betak-1} \frac{1}{\ell^2(n-\ell)^2}
&\le \int_{n/2}^{n-\betak} \frac{\mathrm{d}y}{y^2(n-y)^2} \\
&=
\frac{1}{n^3} \int_{1/2}^{1-\betak/n} \frac{ \mathrm{d}x}{x^2(1-x)^2} \\
&= \frac{1}{n^3} \int_{1/2}^{1-\betak/n} \Bigl(  \frac{1}{x^2}
+ \frac{1}{(1-x)^2}
+ \frac{2}{x} + \frac{2}{1-x} \Bigr) \mathrm{d}x \\
&= \frac{1}{n^3}\Bigl(- \frac{1}{x} +
\frac{1}{1-x} 
+ 2\log x - 2\log (1-x)
\Bigr) \Bigl|_{1/2}^{1-\betak/n} \\
&= \frac{1}{n^3} \Bigl( \frac{n}{\betak} - \frac{1}{1-\betak/n} + 2\log (1-\betak/n)
- 2\log (\betak/n)\Bigr) \\
&\le \frac{1}{n^2\betak} + \frac{2 \log (n/\betak)}{n^3}
\end{align*}
as required.
\end{proof}

We shall also need the computational results contained in
Lemma~\ref{lemma:kappa_comp} below.
There is no particular significance to the choice
of the parameters $n=300$ and $k=15$ in this lemma,  except that
these are convenient numbers to work with,
and they bring the amount
of computational work needed to verify the results
close to a minimum. A similar remark applies to
the use of $n=60$ in parts (ii) and (iii).
(The reader may be interested to
know that the minimum $n$ for which our
proof strategy will work is $n=242$; this requires the
choice $k=14$.)

 Recall that we define $C_\kappa = 13^2
\kappa(S_{13})$. Whenever in this paper we state a
bound as a number written in decimal notation,
it is sharp to five decimal places.

\begin{lemma}
\label{lemma:kappa_comp}
We have
\begin{thmlist}
\item $\kappa(S_{n}) \le C_\kappa / n^2$ for all $n \le 300$;

\item $60 s_{15}(60)
= \frac{158929798034197186400893117108816122671}{833175235266670978029768442202788608000}
< 0.19076$;

\item $n s_{15}(n)\geq (n+1)s_{15}(n+1)$ for $14\leq n\leq 60$;

\item $\sum_{m=0}^{15} \kappa(S_m)
= \frac{4675865182689145531283}{1187508508836249600000}
< 3.93755$;

\item $C_\kappa
= \frac{314540139254371141}{57360633200640000}$
and $5.48355 < C_\kappa < 5.48356$.
\hfill$\Box$
\end{thmlist}

\end{lemma}

\begin{proof} 
All of these results except (i) are routine computations;
(ii) and~(iii) 
follow from the recurrence in \eqref{eq:1rec}, and
(iv) and (v) follow from~\eqref{eq:kappa} in the obvious way.
For (i), we use~\eqref{eq:kappa} to compute the exact value of
$\kappa(S_n)$ for $n\le 80$. For larger $n$
we use the bound in
Proposition~\ref{prop:conjugate_recurrence}, applying it
with whichever choice of $k$ gave the strongest result. For
example, when $n = 81$
the optimal choice of $k$ is $13$, and when $n=300$ it is $39$.
The resulting upper bounds  
easily imply (i).
All of these assertions may be verified using
the computer software mentioned in the introduction to this paper.
\end{proof}

\begin{proof}[Proof of Theorem~\ref{thm:conjugate_uniform_bound}]
We prove the theorem by induction on $n$.
By Lemma~\ref{lemma:kappa_comp}(i) the theorem holds if $n\le 300$,
and so we may assume that $n > 300$.
By Proposition~\ref{prop:conjugate_recurrence} in the
case $k = 15$ we have
\[
\kSn\leq s_{15}(n)^2+ \sum_{\ell=15}^{n} \frac{\kappa(S_{n-\ell})}{\ell^2},
\]
and so
\begin{equation}
\label{eq:kappa_summands} n^2 \kSn \leq n^2 s_{15}(n)^2
+ n^2 \sum_{\ell=n-15}^{n} \frac{\kappa(S_{n-\ell})}{\ell^2}
+ n^2 \sum_{\ell=15}^{n-16} \frac{\kappa(S_{n-\ell})}{\ell^2}.
\end{equation}

It follows from Proposition~\ref{prop:nsk_decreasing} and Lemma~\ref{lemma:kappa_comp}(iii)
that $ns_{15}(n)
\le 60s_{15}(60)$. Hence, by Lemma~\ref{lemma:kappa_comp}(ii), we
have $n^2 s_{15}(n)^2 \le 0.03639$.

Using Lemma~\ref{lemma:kappa_comp}(iv) to bound the second summand
in~\eqref{eq:kappa_summands}
we get
\[ 
 n^2 \sum_{\ell={n-15}}^{n} \frac{\kappa(S_{n-\ell})}{\ell^2}
\le \Bigl( \frac{n}{n-15} \Bigr)^2 \sum_{m=0}^{15}  \kappa(S_m)
\le \bigl(\frac{300}{285}\bigr)^2 \sum_{m=0}^{15} \kappa(S_m) \le 4.36294.
\]

For the third summand in~\eqref{eq:kappa_summands},
we use the inductive hypothesis to get
\[ n^2 \sum_{\ell=15}^{n-16} \frac{\kappa(S_{n-\ell})}{\ell^2}
\le n^2 \sum_{\ell = 15}^{n-16} \frac{C_\kappa}{\ell^2 (n-\ell)^2}.
\]
Using the symmetry in this sum, and then applying
Lemma~\ref{lemma:pfraclemma} in the case $k=15$, we get
\begin{align*}
n^2 \sum_{\ell = 15}^{n-16} \frac{1}{\ell^2 (n-\ell)^2}
 &\le 2n^2 \sum_{\ell = \lceil n/2\rceil}^{n-16} \frac{1}{\ell^2(n-\ell)^2}
 + \frac{n^2}{15^2 (n-15)^2} \\
  &\le 2 \Bigl( \frac{1}{15} + \frac{2\log(n/15)}{n} \Bigr)
 + \frac{1}{15^2} \Bigl( \frac{300}{285} \Bigr)^2  
 \end{align*}
Since $\log (n/15) / n$ is decreasing for $n > 40$, it
follows from the upper bound for $C_\kappa$
in Lemma~\ref{lemma:kappa_comp}(v) that
\begin{align*}
 n^2 \sum_{\ell=15}^{n-16} \frac{\kappa(S_{n-\ell})}{\ell^2}
 &\le  C_\kappa \Bigl( \frac{2}{15} + \frac{4 \log(300/15)}{300}
 + \frac{300^2}{15^2 \cdot 285^2} \Bigr) \\
 &\le 0.97718
\end{align*}
Hence
\[ n^2 \kappa(S_n) \le 0.03639 + 4.36294 + 0.97718  = 5.37651 <
C_\kappa \]
and the theorem follows.
\end{proof}

\section{The limiting behaviour of $\kSn$}
\label{sec:kappa_limit}

Recall that we set $A_\kappa = \sum_{n=1}^\infty \kappa(S_n)$.
In this section, we prove Theorem~\ref{thm:Plimit}
by establishing the two propositions below.

\begin{proposition}\label{prop:Pliminf1}
\[ \liminf_{n \rightarrow \infty} n^2 \kSn \ge A_\kappa.\]
\end{proposition}

\begin{proof}
It follows from the second part
of Proposition~\ref{prop:conjugate_recurrence}
that, if $k > n/2$, then
\[n^2 \kSn \ge \sum_{m =0}^{n-k} \kappa(S_m).\]
Hence taking $k = \lfloor 3n/4\rfloor$ and letting $n \rightarrow \infty$
we see that
\[ \liminf_{n \rightarrow \infty} n^2\kSn \ge \sum_{m=0}^\infty
\kappa(S_m) =A_\kappa.
\]
\end{proof}

\begin{proposition}
\label{prop:Pliminf2}
\[ \limsup_{n \rightarrow \infty} n^2 \kSn \le A_\kappa. \]
\end{proposition}

\begin{proof}
Let $k = \lfloor \frac{n}{\log n}\rfloor$.
By Proposition~\ref{prop:conjugate_recurrence} we have
\[ \kSn \le s_k(n)^2 + \sum_{\ell = k}^n \frac{\kappa(S_{n-\ell})}{\ell^2}. \]
By Proposition~\ref{prop:small_cycles} we have
\[ s_k(n) < \Bigl( \frac{\e}{t} \Bigr)^t \]
where $t = \lfloor \frac{n}{k-1} \rfloor$.
Writing $k = n/\log n + O(1)$ we have
$\lfloor \frac{n}{k-1} \rfloor = (\log n) \Bigl( 1 + \mathrm{O}
\bigl( \frac{\log n}{n} \bigr) \Bigr)$, and so 
\begin{align*}
\log (ns_k(n)) &< \log n + t(1-\log t) \\
              &= 2 \log n - \log n \log \log n + \log
              \Bigl( 1 + \mathrm{O}
\bigl( \frac{\log n}{n} \bigr) \Bigr) \\
& \rightarrow -\infty
\end{align*}
as $n \rightarrow \infty$. Hence $ns_k(n) \rightarrow 0$ as
$n\rightarrow \infty$.

We estimate the main sum in the same way as the proof
of Theorem~\ref{thm:conjugate_uniform_bound}. This gives
\begin{align*} \sum_{\ell = k}^n
\frac{\kappa(S_{n-\ell})}{\ell^2}
&\le
\sum_{\ell=n-k}^n
\frac{\kappa(S_{n-\ell})}{\ell^2} +
 \sum_{\ell = k}^{ n-k-1}
\frac{\kappa(S_{n-\ell})}{\ell^2}
\\
&\le
\sum_{m=0}^{k}
\frac{\kappa(S_{m})}{(n-m)^2} +
\sum_{\ell = k}^{n-k-1}
\frac{C_\kappa}{\ell^2(n-\ell)^2}
 \\
&\leq
\sum_{m=0}^{k} \frac{\kappa(S_m)}{(n-m)^2} +
\sum_{\ell = \lceil n/2\rceil}^{n-k-1}
\frac{2C_\kappa}{\ell^2(n-\ell)^2} +
\frac{C_\kappa}{k^2(n-k)^2}.
\end{align*}
By Lemma~\ref{lemma:pfraclemma},
 the second summand in the equation above is at most
$2C_\kappa \log n / n^3$.
It is easily seen from the identity
\[ \frac{n}{k(n-k)} = \frac{1}{k} + \frac{1}{n-k} \]
that $n^2/k^2(n-k)^2 \rightarrow 0$ as $n \rightarrow \infty$.
Moreover, given any $\epsilon \in \R$ such that $0 < \epsilon < 1$, we have
\[ n^2 \sum_{m=0}^k \frac{\kappa(S_m)}{(n-m)^2} \le \frac{1}{(1-\epsilon)^2}
\sum_{m=0}^{k} \kappa(S_m) \]
for all $n$ such that $1/\log n < \epsilon$.
These remarks show that
\begin{align*}
 \limsup_{n\rightarrow \infty} n^2\kSn &\le
 \limsup_{n\rightarrow \infty}\left(\frac{1}{(1-\epsilon)^2}\sum_{m=0}^{
 \lfloor n /\log n \rfloor}
 \kappa(S_m) \right) \\
&\leq \frac{1}{(1-\epsilon)^2} \sum_{m=0}^{\infty} \kappa(S_{m})\\
&= \frac{A_\kappa}{(1-\epsilon)^2}.
\end{align*}
Since $\epsilon$ was arbitrary, we conclude
that $\limsup_{n\rightarrow \infty}
n^2 \kSn \le A_\kappa$.
\end{proof}

\section{Bounds on $\rho(\Sym(\Omega_n))$}
\label{sec:comm_conj_bounds_for_permutations}

In this section we give the background results needed
to prove the results on $\rho(S_n)$ stated in
Theorems~\ref{thm:conj_comm_uniform_bound}
and~\ref{thm:comm_conj_Plimit}.


Recall that a permutation of a finite set is said to be
\emph{regular} if all its cycles have the same length; a permutation
$\sigma$ acts \emph{regularly} on a subset~$X$ if $X\sigma=X$ and the
restriction $\sigma|_X$ of $\sigma$ to $X$ is regular.
We denote by $r(\ell)$ the probability that
a permutation of $\Omega_\ell$, chosen uniformly at random,
is regular. The following lemma
and proposition are the analogue of Lemma~\ref{lemma:1cycle} 
for regular permutations.

\begin{lemma}
\label{lemma:ell-regular}
For all $\ell \in \mathbf{N}$ we have
\[ \frac{1}{\ell} \le r(\ell) \le \frac{1}{\ell} + \frac{2}{\ell^2}
+ \frac{c}{\ell^3}  \]
where $c = \e^3 / (1-\e/3)$.
\end{lemma}

\begin{proof}
Since an $\ell$-cycle acts regularly, it follows
from Lemma~\ref{lemma:1cycle}(i) that $1/\ell \le r(\ell)$.
The centralizer of a permutation of $\Omega_\ell$ of cycle type
$((\ell/m)^m)$ has order $(\ell/m)^m m!$.
Hence the probability that
a permutation of $\Omega_\ell$, chosen uniformly at random,
has cycle type $((\ell/m)^m)$ is $\frac{m^m}{\ell^m m!}$.  Summing
over all possible values of $m$ we get
\[ r(\ell) = \sum_{m \mid \ell} \frac{m^m}{ \ell^m m!}. \]
To get the claimed upper bound on $r(\ell)$ we estimate the sum
\[ T(\ell) = \sum_{m \mid \ell \atop m \ge 3} \frac{m^m}{ \ell^m m!} \]
using the bound
$m! \ge m^m \e^{-m}$ for $m \ge 1$,
stated earlier in~\eqref{eq:factorial_estimate}.
 This gives
\[
T(\ell) \le \sum_{m \ge 3}  \frac{m^m}{ \ell^m m!}
\le \sum_{m \ge 3} \Bigl( \frac{\e}{\ell} \Bigr)^m
= \frac{\e^3}{\ell^3} \frac{1}{1-\e / \ell}
\le \frac{\e^3}{1-\e / 3} \frac{1}{\ell^3} = \frac{c}{\ell^3}
\]
as required.
\end{proof}

\begin{proposition}
\label{prop:regular_probability}
Let $n \in \mathbf{N}$
and let $L\subseteq \Omega$ be an $\ell$-set. Let $\sigma\in
S_n$ be chosen uniformly at random.  The probability that
$\sigma$ acts regularly on $L$ is $r(\ell) \binom{n}{\ell}^{-1}$.
\end{proposition}

\begin{proof}
We count permutations $\sigma \in S_n$
that act regularly on~$L$ as follows:
there are $r(\ell)\ell!$ choices for the restriction of $\sigma$ to
$L$, and $(n-\ell)!$ choices for the restriction of $\sigma$ to the
complement $\overline{L}$ of $L$.
Hence the probability that $\sigma$ acts regularly on $L$ is
\[ \frac{r(\ell) \ell!(n-\ell)!}{n!} = r(\ell) \binom{n}{\ell}^{-1} \]
as required.
\end{proof}

The connection between permutations that have
conjugates that commute and regular
permutations is elucidated in the next lemma.
The following definition will reduce the amount
of notation required: let $U$
and $V$ be finite sets of the same cardinality,
and let $\pi : U \rightarrow V$ be a bijection.
Let $\sigma$ be a permutation of $U$ and let $\tau$
be a permutation of $V$.
We shall say that $\sigma$ and $\tau$
\emph{have conjugates that commute} if this is the case,
in the ordinary sense, for the permutations $\sigma$
and $\pi^{-1}\tau \pi$ of~$U$.

\begin{lemma}
\label{lemma:commute_condition}
Suppose that
$\sigma$, $\tau\in S_n$ have conjugates that commute and that
the longest cycle length of the cycles
in $\sigma$ and $\tau$ is $m$.
Then there exists an integer~$\ell$
with $m\leq \ell\leq n$ and $\ell$-subsets
$X\subseteq \Omega_n$ and $Y\subseteq \Omega_n$, such that
\begin{thmlist}
\item $\sigma$ acts regularly on $X$, and $\tau$ acts regularly on $Y$;
\item
if $\overline{X}$ and $\overline{Y}$ are the complements
of $X$ and $Y$ in $\Omega_n$, then
the restricted permutations
$\sigma|_{\overline{X}}$ and $\tau|_{\overline{Y}}$
have conjugates that commute.
\end{thmlist}
\end{lemma}

\begin{proof}
 Without loss of generality, we may assume that $\sigma$
contains a cycle of length $m$. Let $M\subseteq\Omega_n$ be the
$m$-set of elements in this cycle.

Since $\sigma$ and $\tau$ have conjugates that commute, there
exists $\pi\in S_n$ such that $\sigma$ and $\pi^{-1}\tau\pi$
commute. Let $A$ be the abelian subgroup of $S_n$ generated by
 $\sigma$ and $\pi^{-1}\tau\pi$, let $X$ be the orbit of $A$
containing $M$, and let $Y=X\pi^{-1}$. Since $\sigma\in A$ we have
$X\sigma=X$, and since $\pi^{-1}\tau\pi\in A$, we have
$X\pi^{-1}\tau\pi=X$, and so $Y\tau=Y$. Define $\ell=|X|=|Y|$, and note
that since $M\subseteq X$ we have $m\leq \ell$.

Since $A$ is abelian and acts transitively on $X$, the restriction of
any element of $A$ to $X$ is regular. In particular, $\sigma$ and
$\pi^{-1}\tau\pi$ act regularly on $X$, and so $\tau$ acts regularly
on $Y$.

The permutations
 $\sigma$ and $\pi^{-1}\tau\pi$ commute, and so their restrictions
$\sigma|_{\overline{X}}$ and $(\pi^{-1}\tau\pi)|_{\overline{X}}$
commute. Since
$(\pi^{-1}\tau\pi)|_{\overline{X}}=\pi^{-1}(\tau|_{\overline{Y}})\pi$,
we see that $\sigma|_{\overline{X}}$ and $\tau|_{\overline{Y}}$ have
conjugates that commute.
\end{proof}

Lemma~\ref{lemma:commute_condition} enables us to prove the following
analogue of Proposition~\ref{prop:conjugate_recurrence}.

\begin{proposition}
\label{prop:conj_commute_recurrence}
For all $n \in \mathbf{N}$
\[
\rho(S_n) \le s_k(n)^2 + \sum_{\ell=k}^{n} r(\ell)^{2} \rho(S_{n-\ell}).
\]
Moreover, if $k$ such that $n/2 < k \le n$, we have
\[
\rho(S_n) \ge \sum_{\ell=k}^n \frac{\rho(S_{n-\ell})}{\ell^2}.
\]
\end{proposition}
\begin{proof}
The proof follows the same pattern as
the
proof of Proposition~\ref{prop:conjugate_recurrence}, and
so we shall give it only in outline.
Given permutations $\sigma$, $\tau \in S_n$ and subsets
$X$, $Y\subseteq\Omega_n$ of the same cardinality $\ell$,
let $E(X,Y)$ be the event
that $\sigma$ acts regularly on $X$, $\tau$ acts regularly on $Y$ and
$\sigma|_{\overline{X}}$ and $\tau|_{\overline{Y}}$ have conjugates
that commute.
If $E(X,Y)$ holds then the
permutations $\sigma|_{\overline{X}}$ and $\tau|_{\overline{Y}}$ are
uniformly distributed over
the symmetric groups on $\overline{X}$ and $\overline{Y}$
respectively. Hence it follows from
Proposition~\ref{prop:regular_probability}
that 
\[
\P(E(X,Y))=\binom{n}{\ell}^{-2} r(\ell)^2 \rho(S_{n-\ell}).
\]

 Lemma~\ref{lemma:commute_condition} implies that
if $\sigma$, $\tau \in S_n$ have conjugates that
commute then either all of the cycles in $\sigma$ and $\tau$
have length strictly less than $k$, or at least one
of $\sigma$ and $\tau$ has a cycle of length at least $k$ and there exist
sets $X$ and $Y$ of cardinality $\ell \ge k$ on which $X$ and $Y$ both
act regularly.  (These events are not mutually exclusive.) Thus
\begin{align*}
\rho(S_n) &\leq
s_k(n)^2+\sum_{\ell=k}^n\sum_{|X|=\ell}\sum_{|Y|=\ell}\P(E(X,Y))\\
&\leq s_k(n)^2 + \sum_{\ell=k}^{n} r(\ell)^{2} \rho(S_{n-\ell}),
\end{align*}
as required by the first part of the Proposition.

The events $E(X,Y)$ cannot be
used to establish the lower bound of the
proposition, since they are not disjoint, even when
$\ell>n/2$. Instead, we define $F(X,Y)$ to be the event
that $\sigma$ acts as an $\ell$-cycle on $X$, $\tau$ acts as an
$\ell$-cycle on $Y$ and $\sigma|_{\overline{X}}$ and
$\tau|_{\overline{Y}}$ have conjugates that commute.
A similar argument to that used
for $E(X,Y)$, using Proposition~\ref{prop:regular_probability}
in place of Lemma~\ref{lemma:1cycle}(i), shows that
\[ \P(F(X,Y)) = \binom{n}{\ell}^{-2}
\ell^{-2} \rho(S_{n-\ell}).\]

The events $F(X,Y)$ with $|X|=|Y| > n/2$ are disjoint since a
permutation can contain at most one cycle of length strictly greater than $n/2$. Moreover,
if an event $F(X,Y)$ occurs, then $\sigma$ and $\tau$ have
conjugates that commute. It follows that whenever $k>n/2$,
\begin{align*}
\rho(S_n) &\ge \sum_{\ell=k}^n\sum_{|X|=\ell}\sum_{|Y|=\ell}\P(F(X,Y))\\
& = \sum_{\ell=k}^n \frac{\rho(S_{n-\ell})}{\ell^2}
\end{align*}
as required.
\end{proof}

\section{Proofs of Theorems~\ref{thm:conj_comm_uniform_bound}
and~\ref{thm:comm_conj_Plimit}}
\label{sec:rho_proofs}

We need the following computational lemma, to which
remarks similar to those made before Lemma~\ref{lemma:kappa_comp} apply.

\begin{lemma}\label{lemma:rho_comp}
We have
\begin{thmlist}
\item $\rho(S_{n}) \le C_\rho / n^2$ for all $n \le 180$;

\item $180 s_{30}(180) < 0.00247$;

\item $ns_{30}(n)\geq (n+1)s_{30}(n+1)$ for $29\leq n\leq 180$;

\item $\sum_{m=0}^{30} \rho(S_m)
< 6.11806$;

\item $C_\rho
= \frac{5805523}{508032}$
and $11.42747 < C_\rho < 11.42748$.

\end{thmlist}

\end{lemma}

\begin{proof}
For (i) we compute
the exact value of $\rho(S_n)$ when $n \le 35$.
For $n \ge 36$
the result follows from
the bound in
Proposition~\ref{prop:conj_commute_recurrence},
again using whichever choice of $k$ gives the strongest result.
Parts (ii), (iii), (iv) and (v) are proved by the same
methods used in Lemma~\ref{lemma:kappa_comp}.
\end{proof}

\noindent
\begin{proof}[Proof of Theorem~\ref{thm:conj_comm_uniform_bound}]
The proof proceeds along the same lines as the
proof of
Theorem~\ref{thm:conjugate_uniform_bound}.
By Lemma~\ref{lemma:rho_comp} the theorem holds if $n \le 180$, and
so we may assume, inductively, that $n\ge 180$.
Proposition~\ref{prop:conj_commute_recurrence} implies that
\begin{equation}\label{eq:rho_summands}
 n^2 \rho(S_n) \le n^2 s_{30}(n)^2 + n^2 \!\!  \sum_{\ell = n-30}^n
\!\! r(\ell)^2 \rho(S_{n-l}) +
 n^2  \sum_{\ell = 30}^{n-31} r(\ell)^2 \rho(S_{n-\ell}).
\end{equation}
By Proposition~\ref{prop:nsk_decreasing} and Lemma~\ref{lemma:rho_comp}(ii) and~(iii)
we have $ns_{30}(n) \le 180s_{30}(180) \le 0.00247$ for all $n\ge 180$.
Hence $n^2 s_{30}(n)^2 \le 0.00001$.
To deal with the other two summands, it will be useful to introduce
the function
$R(\ell) = 1 + 2/\ell + c/\ell^2$. Note
that, by Lemma~\ref{lemma:ell-regular},
we have $r(\ell) < R(\ell)/\ell$ for all $\ell \in \mathbf{N}$.
For the second summand, we have
\begin{align*}
n^2 \sum_{\ell=n-30}^n
r(\ell)^2 \rho(S_{n-\ell}) &\le
\Bigl( \frac{n}{n-30}\Bigr)^2 R(n-30)^2
\sum_{m=0}^{30} \rho(S_m) \\
&\le \Bigl( \frac{180}{150} \Bigr)^2 R(150)^2
\sum_{m=0}^{30} \rho(S_m) \\ &\le 9.21704   
\end{align*}
where we have used Lemma~\ref{lemma:rho_comp}(iv) and the obvious fact
that $R(\ell)$ is a decreasing function of~$\ell$.
For the third summand, we use the inductive hypothesis to obtain
\[ n^2 \sum_{\ell=30}^{n-36} r(\ell)^2\rho(S_{n-\ell})
\le  R(30)^2 \sum_{\ell = 30}^{n-31} \frac{n^2 C_\rho}{\ell^2 (n-\ell)^2}.
\]
Arguing as in the proof of Theorem~\ref{thm:conjugate_uniform_bound}
we find that
\begin{align*}
n^2 \sum_{\ell=30}^{n-36} r(\ell)^2 \rho(S_{n-\ell})
&\le  R(30)^2 C_\rho \Bigl( \frac{2}{30} + \frac{4 \log (180/30)}{180}
+ \frac{180^2}{30^2 \cdot 150^2} \Bigr) \\
&\le 2.10126.
\end{align*}
Hence
\[ n^2 \rho(S_n) \le 0.00001 + 9.21704 +  2.10126 = 11.31831 \le C_\rho \]
and the theorem follows.
\end{proof}

\begin{proof}[Proof of Theorem~\ref{thm:comm_conj_Plimit}]
One may prove that
\[ \liminf_{n \rightarrow \infty} n^2 \rho(S_n) \ge A_\rho\]
by using the lower bound in Proposition~\ref{prop:conj_commute_recurrence}
and the same method as Proposition~\ref{prop:Pliminf1}.
For the upper limit, it follows from
Proposition~\ref{prop:conj_commute_recurrence},
by the same arguments used in Proposition~\ref{prop:Pliminf2} that
\[
\limsup_{n \rightarrow \infty} n^2 \rho(S_n) =
\limsup_{n \rightarrow \infty} \; n^2 \! \sum_{\ell=\lfloor n/\log n\rfloor}^nr(\ell)^2 \rho(S_{n-\ell}).
\]
Let $\epsilon \in \R$ be given with $0 < \epsilon < 1$. 
Lemma~\ref{lemma:ell-regular} shows that $r(\ell)\leq
(1+\epsilon)/\ell$ for all sufficiently large $\ell$. So
provided $n$ is sufficiently large, we have
\[
\sum_{\ell=\lfloor n/\log n\rfloor}^nr(\ell)^2\rho(S_{n-\ell})\leq (1+\epsilon)^2\sum_{\ell=\lfloor n/\log n\rfloor}^n \frac{\rho(S_{n-\ell})}{\ell^2}.
\]
It now follows, as in the proof of Proposition~\ref{prop:Pliminf2}, that
\[
n^2\sum_{\ell=\lfloor n/\log n\rfloor}^n \frac{\rho(S_{n-\ell})}{\ell^2}
\leq \frac{A_\rho}{(1-\epsilon)^2}
\]
whenever $n$ is sufficiently large, and so
\[
\limsup_{n \rightarrow \infty} n^2 \rho(S_n)
\leq A_\rho \frac{(1+\epsilon)^2}{(1-\epsilon)^2}.
\]
The theorem follows.
\end{proof}

\section{Final remarks and open problems}
\label{sec:final}

\subsection*{Remarks on $\kappa(G)$}
There are a number of interesting open problems that concern the
 spectrum of values taken by $\kappa(G)$ as $G$ varies
over all finite groups. It is clear that if $G$ and $H$ are
finite groups then \hbox{$\kappa(G \times H) = \kappa(G) \times \kappa(H)$},
and so the spectrum is closed under multiplication.
Moreover, generalizing part of
Theorem~\ref{thm:uppergap}, it is not hard to prove that
if there exists a Frobenius group with point stabiliser
$H$ and Frobenius kernel $K$ then the spectrum has a limit
point at $\kappa(H) - \frac{1}{|H|^2}$.
Indeed, if $G$ is such
a group then, by Lemma~\ref{lemma:frobenius},
\[ \kappa(G) = \frac{1}{|G|^2} + \frac{1}{|H|} \Bigl( \kappa(K) - \frac{1}{|K|^2} \Bigr)
+ \Bigl(  \kappa(H) - \frac{1}{|H|^2} \Bigr). \]
We may construct Frobenius groups
of arbitrarily large cardinality,
each with point stabiliser~$H$, by making $H$ act in the obvious way
on the direct product $K^m$ for $m \in \mathbf{N}$.
Since $\kappa(K^m) = \kappa(K)^m \rightarrow 0$ as $m \rightarrow \infty$,
this family gives the claimed limit point.

It is natural to ask whether every limit point of $\kappa(G)$
is explained by a family of Frobenius groups. If this is false,
one might still ask whether there are irrational limit points.

We may also ask whether we can say anything about the
value of $\kappa(G)$ for a typical group~$G$. A framework
for this question is as follows:
for a fixed $\alpha\in [0,1]$, let $f_\alpha(n)$ be the number of isomorphism classes of groups $G$ of order $n$ with $\kappa(G)\geq \alpha$. What can be said about the growth of this function?

\subsection*{Remarks on $\rho(G)$}
It was shown at the end of \cite{BritnellWildon} that
if $G$ is a finite group and $g \in G$, then $g \in Z(G)$
if and only if,  for each $h \in G$,
there is a conjugate of $g$ which commutes with $h$.
 It follows that $\rho(G) = 1$
if and only if $G$ is abelian.

It is not hard to see that
can be no `upper gap' result on $\rho(G)$ analogous
to Theorem~\ref{thm:uppergap}.
Indeed, if $G$ is a Frobenius group
with abelian point stabiliser $H$ and Frobenius kernel $K$
then, by~\eqref{eq:decomp} in the proof of Lemma~\ref{lemma:frobenius},
any two elements in $G \backslash K$ have conjugates
that commute. Hence
\[ \rho(G) > \Bigl( 1 - \frac{1}{|H|} \Bigr)^2 . \]
It follows on considering
the Frobenius groups of order $p(p-1)$ for
large primes $p$, that the set of values of $\rho(G)$
has a limit point at $1$.

It is obvious that if $G$ is
any finite group then $\rho(G) \ge \kappa(G)$, and so
the `lower gap' result in Theorem~\ref{thm:lowergap}
also applies to $\rho$. We leave it to the reader to check
that $\rho(G) = \kappa(G)$ if and only if~$G$ satisfies
the second Engel identity $[x,[x,g]] = 1$ for all $x,g\in G$.
Such groups are necessarily nilpotent, by an important theorem
of Zorn \cite{Zorn}.

More mysteriously, we note that with surprising frequency,
$|G|\rho(G)$ is an integer.
For instance, this is the case for all groups of order $< 54$.

Finally, we remark that all of the
algorithms we know of which check whether two
given permutations are conjugate to commuting elements are inefficient in the worst case (though we know of algorithms that are efficient for most pairs of permutations in practice). Does an efficient algorithm for this problem exist?

\subsection*{A unified framework.}

In \cite[\S 1.2]{Dixon2002}, Dixon considers the probability that a
particular instance of a law will be found to hold in a finite group.
For example, the probability that the law $xy=yx$ holds
in a finite group
$G$ is the commuting probability $\cp(G)$.
If we generalize the idea of a law to allow
existential quantifiers then we obtain a unified
framework for our results on $\kappa$ and $\rho$; the
relevant formulae in the first order language of groups are
$ \exists g \; x^g = y$ and $\exists g\; x^g y = yx^g$,
respectively. If $N$ is a normal subgroup
of $G$ and a probability is defined in this way, then the probability associated with $G/N$ is bounded below by the corresponding probability associated with $G$.
Thus a slightly weaker version of
Lemma~\ref{lemma:quo} holds in much greater generality.

\subsection*{Acknowledgements}
The authors would like to thank the anonymous referee for a careful reading
of this paper.


\begin{thebibliography}{99}

\bibitem{ArratiaEtAl} R. Arratia, A. D. Barbour and S. Tavar{\'e},
\emph{Logarithmic combinatorial structures: a probabilistic approach},
European Mathematical Society, 2003.

\bibitem{Bertram} E. A. Bertram, \emph{Lower bounds for the number of conjugacy classes in finite groups},
Ischia Group Theory 2004, Contemporary Mathematics 402, American Mathematical
Society, Providence, RI, 2006, 95�-117.

\bibitem{BritnellWildon} J. R. Britnell and M. Wildon, \emph{Commuting
elements in conjugacy classes: An application of Hall's Marriage
Theorem to group theory}, J. Group. Theory 12 (2009) 795--820.

\bibitem{BritnellWildonGL} J. R. Britnell and M. Wildon,
\emph{On types and classes of commuting matrices over finite fields},
J. Lond. Math. Soc. 83 (2011) 470--492.

\bibitem{CaminaCamina} A. R. Camina and R. D. Camina, \emph{The influence of conjugacy class sizes
on the structure of finite groups: A survey}, Asian-European Journal of Mathematics, to appear.

\bibitem{Dixon1973} J. D. Dixon, \emph{Solution to Problem 176}, Canadian Mathematical Bulletin, 16 (1973), 302.

\bibitem{Dixon2002} J. D. Dixon, \emph{Probabilistic group theory},
C. R. Math. Acad. Sci. Canada, 24 (2002), 1--15.

\bibitem{ErdosTuranIV} P. Erd\H{o}s and P. Tur\'an, \emph{On some problems of a statistical
group theory IV}, Acta Mathematica Academiae Scientiarum Hungaricae 19 3--4 (1968), 413--435.

\bibitem{FeitThompson}
W. Feit and J. G. Thompson, \emph{Groups which contain a self-centralizing
subgroup of order $3$}, Nagoya Math. J. 21 (1962), 185--97.

\bibitem{FlajoletEtAl}
P. Flajolet, {\'E}. Fusy, X. Gourdon, D. Panario and N.
Pouyanne, \emph{A Hybrid of Darboux's Method and Singularity Analysis in Combinatorial Asymptotics}, Electr. J. Comb. 13 (2006).

\bibitem{Goncharov1} V. Goncharov, \emph{Sur la distribution des cycles dans les permutations},
C. R. (Doklady) Acad.\ Sci.\ URSS (N.S.) 35 (1942), 267--269.

\bibitem{Goncharov2} V. Goncharov, \emph{Du domaine d'analyse combinatoire}, Izv.\ Akad.\ Nauk.\
SSSR Ser.\ Mat.\ 8 (1944), 3--48.

\bibitem{Gorenstein}
D. Gorenstein, \emph{Finite Groups, second edition}, Chelsea Publishing Company, New York, 1980.

\bibitem{Gustafson} W. H. Gustafson, \emph{What is the probability that two group elements
commute?}, Amer. Math. Monthly 80 (1973), 1031--1034.

\bibitem{Herzog}
M. Herzog, \emph{On centralizers of involutions}, Proc. Amer. Math. Soc. 22
(1969), 170--174.

\bibitem{Isaacs}
I. M. Isaacs, \emph{Character theory of finite groups}, Associated Press,
(1976).


\bibitem{Karamata}
J. Karamata, \emph{Sur une in{\'e}galit{\'e} relative aux fonctions convexes}, Publications math{\'e}matiques de l'Universit{\'e} de Belgrade,
1 (1932), 145--148.

\bibitem{Keller} T. M. Keller,
\emph{Finite groups have even more conjugacy classes},
Israel J. Math. 181 1 (2011), 433--444.

\bibitem{Laffeyp}
T. J. Laffey, \emph{The number of solutions of $x^p=1$ in a finite
group}, Math. Proc. Cambridge Philos. Soc. 80 (1976), 229--231.

\bibitem{Laffey3}
T. J. Laffey, \emph{The number of solutions of $x^3=1$ in a $3$-group},
Math. Z. 149 (1976), 43--45.

\bibitem{Lescot} P. Lescot, \emph{Isoclinism classes and commutativity degrees of finite groups},
J. Algebra 177 3 (1995), 847--869.

\bibitem{OEIS}
OEIS Foundation Inc. (2011), \emph{The On-Line Encyclopedia of Integer
Sequences}, \url{http://oeis.org}.


\bibitem{PeytonJones}
S. Peyton-Jones \emph{et al}, \emph{The {Haskell} 98 Language and Libraries: The Revised Report}, Journal of Functional Programming 13 (2003), 0--255.

\bibitem{Pyber} L. Pyber, \emph{Finite groups have many conjugacy classes}, J. London Math.
Soc. (2) 46 2 (1992), 239-�249.

\bibitem{SheppLloyd} L. A. Shepp and S. P. Lloyd, \emph{Ordered cycle lengths in a random permutation},
Trans.\ Amer.\ Math.\ Soc.\ 121 (1966) 340--357.

\bibitem{Suzuki} M. Suzuki, \emph{On finite groups containing an element
of order four which commutes only with its powers}, Illinois J. Math
3 (1959), 255--271.

\bibitem{Thompson}
J. Thompson. \emph{Finite groups with fixed-point-free automorphisms of prime order},
\emph{Proc.\ Nat.\ Acad.\ Sci.\ U.S.A.} \textbf{45} (1959), 578--581.

\bibitem{Wall}
C. T. C.\  Wall, \emph{On groups consisting mostly of
  involutions}, Proc. Camb. Phil. Soc. \textbf{67} (1970), 251--262.

\bibitem{GWall} G. E. Wall, \emph{On Hughes' $H_p$ problem},
Proc. Internat. Conf. Theory of Groups, (Canberra, 1965), Gordan
  and Breech, New York, 1967, 357-362.

\bibitem{Mathematica} Wolfram Research, Inc., \emph{Mathematica}, Version 7.0, Champaign, IL (2008).

\bibitem{Zorn} M. Zorn, \emph{Nilpotency of finite groups},
Bull. Amer. Math. Soc. \textbf{42} (1936), 485-486.

\end{thebibliography}
\end{document}